\documentclass{siamart1116}

\usepackage[l2tabu, orthodox]{nag}
\usepackage{microtype}
\usepackage{graphicx}
\usepackage{subcaption}
\usepackage{epstopdf}
\usepackage{algorithmic}
\usepackage{amsfonts,amsmath,amssymb,amsfonts}
\usepackage{upgreek, dsfont, tipa, color, inconsolata, framed,
  mathtools, sidecap, graphics, graphicx, wrapfig, overpic, hyperref,
  mathrsfs, textcomp, verbatim, hyperref}
\usepackage{booktabs}
\usepackage[group-separator={,}]{siunitx}
\usepackage{pgfplots}
\pgfplotsset{compat=1.13}

\usetikzlibrary{backgrounds}
\usepackage{tikzscale}
\usepackage{todonotes}
\graphicspath{{./img/}}

\usepackage[square,numbers,sort&compress]{natbib}

\usetikzlibrary{calc}

\ifpdf
\DeclareGraphicsExtensions{.eps,.pdf,.png,.jpg}
\else
\DeclareGraphicsExtensions{.eps}
\fi

\newsiamremark{myremark}{Remark}
\newsiamthm{myassumption}{Assumption}

\numberwithin{theorem}{section}

\DeclareMathOperator*{\argmin}{arg\,min}

\newcommand{\vect}[1]{\boldsymbol{#1}}

\newcommand{\vecbracks}[1]{\left[#1\right]}

\newcommand{\rvX}{X_t}
\newcommand{\rvXDiscrete}[1]{X_{#1}}

\newcommand{\Xv}{\vect{X}}
\newcommand{\Yv}{\vect{Y}}

\newcommand{\dd}{\text{d}}

\newcommand{\De}{\Delta}

\newcommand{\MeaSp}{\mathcal{M}}

\newcommand{\MeaPr}{\mu}

\newcommand{\SigAlg}{\mathcal{B}}
\newcommand{\MeaPrDat}{\MeaPr_\LimDat}

\newcommand{\Koopd}{\mathcal{K}}
\newcommand{\ProjKoopd}[1]{\Koopd_{#1}}
\newcommand{\ProjKoopdDat}{\ProjKoopd{\MeaPrDat}}

\newcommand{\DiffGen}{\mathcal{L}}

\newcommand{\LimDat}{T}
\newcommand{\LimBasis}{N}
\newcommand{\LimEig}{J}

\newcommand{\Drift}{a}
\newcommand{\Vol}{b}

\newcommand{\FlowX}[1]{\rvX^{#1}}
\newcommand{\FlowXDiscrete}[2]{\rvXDiscrete{#1}^{#2}}
\newcommand{\Transpose}[1]{#1^\intercal}

\newcommand{\Koopdmat}[1]{A_{#1}}
\newcommand{\KoopdmatDat}{\Koopdmat{\MeaPrDat}}
\newcommand{\Transmat}[1]{K_{#1}}
\newcommand{\TransmatDat}{\Transmat{{\MeaPrDat}}}
\newcommand{\EDMDTransmat}{\hat{K}_{\LimDat}}
\newcommand{\Massmat}[1]{M_{#1}}
\newcommand{\MassmatDat}{\Massmat{{\MeaPrDat}}}
\newcommand{\Lmat}[1]{L_{#1}}

\newcommand{\EDMDmat}{\hat{A}_{\LimDat}}
\newcommand{\EDMDop}{\hat{\Koopd}_{\LimDat}}

\newcommand{\gf}{g}
\newcommand{\vgf}{\vect{\gf}}

\newcommand{\FunSpGInf}{\mathcal{G}}
\newcommand{\FunSpGFin}{G_{\BBasis}}
\newcommand{\TargetField}{\mathbb{R}}

\newcommand{\Proj}{P}
\newcommand{\norm}[1]{\left\lVert#1\right\rVert}

\newcommand{\Basis}{\psi}
\newcommand{\BBasis}{\vect{\psi}}

\newcommand{\Param}{\vParam}
\newcommand{\vParam}{\theta}
\newcommand{\ParamSet}{\Theta}

\makeatletter \def\input@path{{./img/},{./}} \makeatother

\newcommand{\TheTitle}{Operator Fitting for Parameter Estimation of
  Stochastic Differential Equations}
\newcommand{\ShortTitle}{Operator
  Fitting for Parameter Estimation of SDEs}
\newcommand{\TheAuthors}{Asbj{\O}rn N. Riseth, and Jake
  P. Taylor-King}

\headers{\ShortTitle}{\TheAuthors}

\title{{\TheTitle}
}

\author{
  Asbj{\O}rn N. Riseth\thanks{Mathematical Institute,
    University of Oxford, Oxford, OX2 6GG, UK} \and
  Jake P. Taylor-King\footnotemark[1]
  \thanks{Department of Integrated
    Mathematical Oncology, H. Lee Moffitt Cancer Center and Research
    Institute, Tampa, FL, USA
    \newline{\bf Contact email:}
    taylor-king@imsb.biol.ethz.ch}\newline }

\usepackage{amsopn}

\ifpdf
\hypersetup{
  pdftitle={\TheTitle},
  pdfauthor={\TheAuthors}
}
\fi

\begin{document}

\maketitle

% REQUIRED
\begin{abstract}
  Estimation of parameters is a crucial part of model
  development. When models are deterministic, one can minimise the
  fitting error; for stochastic systems one must be more
  careful. Broadly, parameterisation methods for stochastic dynamical
  systems fit into maximum likelihood estimation- and method of
  moment-inspired techniques. We propose a method where one matches a
  finite dimensional approximation of the Koopman operator with the
  implied Koopman operator as generated by an extended dynamic mode
  decomposition approximation. One advantage of this approach is that
  the objective evaluation cost can be independent the number of
  samples for many dynamical systems. We test our approach on two
  simple systems in the form of stochastic differential equations,
  compare to benchmark techniques, and consider limited
  eigen-expansions of the operators being approximated. Other small
  variations on the technique are also considered, and we discuss the
  advantages to our formulation.
\end{abstract}

% REQUIRED

% REQUIRED
%\begin{AMS}
%  47N30, % Operator theory / Applications in probability theory and statistics
%  60H10, % Stochastic analysis / Stochastic ordinary differential equations
%  62F99, % Statisics / None of the above
%  65D99. % Numerical approx and geometry / None of the above
%\end{AMS}

%%%%%%%%%%%%%%%%%%%%%%%%%%%%%%%%%%%%%
\section{Introduction}
%%%%%%%%%%%%%%%%%%%%%%%%%%%%%%%%%%%%%

% Stochastic Systems Extension of Deterministic Systems:} }}

% Stochastic Systems Extension of Deterministic Systems:} }}
In multiple application areas, such as physics and biology, noise
plays an important role in the system
dynamics~\cite{Gardiner_1994,VanKampen_1992}. One way to include noise
into the dynamics is to add suitable stochastic terms to ordinary
differential equations (ODEs), which leads to so-called stochastic
differential equations (SDEs)~\cite{Pavliotis_2016}.  Possible
dynamics of SDEs compared to ODEs are then greatly enriched due to the
presence of noise, making the SDE suitable to capture intriguing
noise-induced phenomena, such as noise-induced switching,
oscillations, and
focusing~\cite{Gardiner_1994,VanKampen_1992,Erban_2007}.

% Parameter estimation hard} }

With a family of parameterised deterministic dynamical systems, one
typically chooses parameters such that a suitable objective function
is minimised. In the case of continuous state-space dynamical systems,
usually one minimises the mean squared error between model prediction
and observed data; this can be achieved via nonlinear least
squares~\cite{Peifer_2007, Ramsay_2007, Bock_2007}.  However, methods
useful in the deterministic setting are unsuitable when applied to
dynamical systems with intrinsic noise --- especially when
noise-induced phenomena are present.

% Forward / backward parameter fitting.

Many well behaved dynamical systems have an associated forward and
backward interpretation~\cite{klus2016On, klus2017data}. Depending on
the interpretation used, this can lead to different numerical methods
for parameter estimation~\cite{Hartig_2011, hurn2007seeing}. The
forward interpretation describes the time-evolution of the probability
that the system is in some state, and is known as the
Perron--Frobenius operator (PFO). The PFO naturally links to maximum
likelihood estimation (MLE), where one selects parameters for a model
such that the probability of observed data being realised (by said
model) is maximised. The backward interpretation is adjoint to the
forward interpretation and describes the time-evolution of
expectations, known as the Koopman\footnote{ The Koopman operator for
  random dynamical systems is also called the stochastic Koopman
  operator, or, in probability terminology, the transition operator.}
operator (KO). In the method of moments (MM), parameters are chosen to
match theoretical expected values of a model to sample mean values as
calculated from an observed data set. Naturally, solving the PFO or KO
via numerical scheme and implementing a minimisation procedure to find
the optimal parameter choice can be computationally intensive, so
various methods have been proposed, for example, Monte Carlo
simulation, or approximate Bayesian computation \cite{Golightly_2011,
  Hines_2015, Picchini_2014}.

% What is EDMD?}}
Recently, using data to numerically reconstruct the Koopman operator
and the Perron--Frobenius operator (and their eigendecompositions)
have become a popular area of study~\cite{klus2016On,klus2017data}.
One of these methods is known as Extended Dynamic Mode Decomposition
(EDMD)~\cite{williams2015data, Kutz_2016}, which uses basis functions
to project the data into a higher dimensional space, where it is
assumed dynamics are linear. Therefore one can propagate forward
nonlinear models in a linear fashion. Depending on the basis functions
used, EDMD methods can scale very well with dimension and efficient
numerical schemes have been proposed. For example, Williams
\emph{et. al.} \cite{rowley2009spectral} used an Arnoldi type
algorithm for analysing data from direct numerical simulations of a
turbulent flow on a $256\times201\times144$ grid, and Tu
\emph{et. al.} \cite{tu2014dynamic} presented a SVD based algorithm
that was applied to analysis of an incompressible Navier-Stokes
generated flow on a $1500\times 500$ grid.

% Current approaches:}}

To avoid the limitless possible stochastic dynamical systems we could
consider, we restrict ourselves to SDEs. In the case of SDEs, the PFO
defines the Fokker--Planck equation and the KO defines the Kolmogorov
backward equation. In the review by Hurn
\emph{et. al.}~\cite{hurn2007seeing}, two classification terms were
identified as alternatives to exact MLE\footnote{Exact MLE was the
  term used to describe the case where the transition probability had
  a known analytic expression.}. These were: likelihood based
procedures, essentially trying to estimate the likelihood function
using a numerical scheme; and the obscurely named ``sample DNA
matching'' procedures, where one tries to match some feature(s) of the
model to some feature(s) of the data --- essentially accounting for
all other parameter estimation methods.

% What does our approach do?:}}

Our method involves the following steps: we calculate the EDMD matrix
as implied by the data; we use the same basis functions as the EDMD
matrix to build a matrix representation of the Koopman operator; and
we then choose parameters such that these matrices are as close as
possible (under some norm).
We prove that, under relevant conditions, EDMD estimates
the correct Koopman operator with a mean-square error that scales with
the inverse of the number of data points.
A similar approach of matching matrix representations of operators
within the Koopman framework have also been proposed
in~\cite{mauroy2016linear,mauroy2017koopman}, where they focus on
ODEs with polynomial nonlinearities.
Numerical tests indicate convergence of
the root mean squared error as the amount of data increases.
Fundamental to EDMD is that one can decompose the (approximate)
Koopman operator into eigenfunctions. For one of our numerical
examples we use this idea to show that for large quantities of data, a
limited eigen-expansion can provide better parameter estimates when
compared to the full matrix representation. Our method is neither an
MLE based method nor an MM based method, but can be placed in the
category of sample DNA matching methods as described by Hurn
\emph{et. al.}~\cite{hurn2007seeing}.

% Hows does it supersede previous approaches?}}

Using our EDMD-based approach, we can carry out computationally cheap
parameterisations of SDEs (depending on the choice in basis
functions). Our method is comparable to other standard techniques for
SDEs, and we also mention variations on our method. The method is
general in that it should be clear how to adapt the approach to other
dynamical systems.

% Layout of paper}}

The paper
is organised as follows. Our algorithm and its theoretical motivation
are in \Cref{sec:sdeintro}, and numerical experiments are in
\Cref{sec:num_examples}. We consider variations of our algorithm and
similarities to existing methods in \Cref{sec:variations_algo}, and
discuss further work in \Cref{sec:conclusion}.

%%%%%%%%%%%%%%%%%%%%%%%%%%%%%%%%%%%%%
\section{Markov operators of SDEs}\label{sec:sdeintro}
%%%%%%%%%%%%%%%%%%%%%%%%%%%%%%%%%%%%%
We are concerned with parameter estimation for one-dimensional
autonomous SDEs of the form\footnote{In
  this paper, we only consider one-dimensional SDEs, although higher
  dimensional SDEs follow naturally.}
\begin{equation}\label{eq:general_sde}
  \dd \rvX= \Drift( \rvX; \vParam)\, \dd t + \Vol( \rvX; \vParam )\, \dd W_t \, .
\end{equation}
Here, $W_t$ is a standard Brownian motion, and $\Drift(x;\vParam)$ and
$\Vol(x;\vParam)$ are drift and volatility functions parameterised by
$\vParam$ from some parameter set $\ParamSet$.  To the
SDE~\eqref{eq:general_sde} there is an associated transition operator
semigroup that describe the expected evolution of functions of the
state. We follow the DMD-literature and refer to these operators as
Koopman operators, due
to~\cite{koopman1931hamiltonian,lasota1985probabilistic}.  For our
purposes, assume that $\Drift(x;\vParam)$ and $\Vol(x;\vParam)$ are
well behaved, so that a strong solution to the SDE exists for relevant
initial conditions.  Given a data set, believed to be generated by a
process of the form~\eqref{eq:general_sde}, we want to obtain a
parameter estimate $\hat{\vParam}$. When using synthetic data
generated by the SDE with parameter $\vParam^*$, we wish for our
estimate to match the true value.

Our approach builds on Koopman theory and EDMD to find
$\hat{\vParam}$. The core idea of the proposed algorithm is to compare
an EDMD operator, approximated from data, with the Koopman operator
associated with~\eqref{eq:general_sde}. We can also place the
Perron--Frobenius operator in the context of our algorithm; however,
as EDMD is best understood as an approximation to the Koopman operator
we focus on the backward interpretation (with occasional mentions to
the forward interpretation). We start by introducing the relevant
theory, before formalising our algorithm. The contents of
\Cref{subsec:edmd_description,subsec:edmd_koop_conn} follows
the exposition of Korda and Mezi{\'c}~\cite{korda2018convergence}.

%%%%%%%%%%%%%%%%%%%%%%%%%%%%%%%%%%%%%
\subsection{The Koopman Operator}\label{subsec:koop_description}
%%%%%%%%%%%%%%%%%%%%%%%%%%%%%%%%%%%%%

Denote by $\MeaSp\subset \mathbb{R}$ the state space of the
SDE~\eqref{eq:general_sde}.  Define $\FlowX{x}$ to be the solution to
the SDE at time $t$, with initial condition $X_0=x$.
\begin{definition}
  The time $t\geq 0$ \emph{Koopman operator} on functions
  $g:\MeaSp\to\TargetField$, is given by
  \begin{equation}\label{eq:def_koop_t}
    \Koopd^tg(x) = \mathbb{E}_W\left[ g(\FlowX{x}) \right],
  \end{equation}
  where the expectation is taken with respect to the paths of the
  Brownian motion $W_t$.
\end{definition}
We require that the domain of $\Koopd^t$ is such that $g(\FlowX{x})$
is integrable with respect to the paths of the Brownian motion. Note
that $\Koopd^t$ is linear on the space of functions.  In the case of
an ODE, that is, when
$b(x;\theta)\equiv 0$, the solution $\FlowX{x}$ is deterministic and
the Koopman operator is just $\Koopd^tg(x)=g(\FlowX{x})$.

For completeness we note that the Perron--Frobenius operator can be
viewed as the adjoint of the Koopman operator.  Associate a
probability space $(\MeaSp,\SigAlg,\MeaPr)$ to the state space of the
SDE~\eqref{eq:general_sde}.  The operators are then adjoint under the
duality pairing between $L^1(\MeaSp,\MeaPr)$ and
$L^\infty(\MeaSp,\MeaPr)$, the spaces of integrable, and essentially
bounded functions respectively. See, for example Klus
\emph{et. al.}~\cite{klus2016On}, for a broader discussion of the
duality between the two operators and their numerical approximations.

%%%%%%%%%%%%%%%%%%%%%%%%%%%%%%%%%%%%%
\subsection{Extended Dynamic Mode
  Decomposition}\label{subsec:edmd_description}
%%%%%%%%%%%%%%%%%%%%%%%%%%%%%%%%%%%%%

Suppose we are given snapshots of input-output data
\begin{equation}
  \Xv = \vecbracks{x_1,\dots,x_\LimDat},\quad \Yv=\vecbracks{y_1,\dots,y_\LimDat},
\end{equation}
where $y_j$ is a realisation of $\FlowXDiscrete{t}{x_j}$ for some
fixed time-step $t > 0$.  If the data is a sample path of the solution
to the SDE, we have that $x_{j+1}=y_j$ for $j=1,\dots,\LimDat-1$, that
is, $x_{j+1} = \FlowXDiscrete{jt}{x_1}$.  In the remainder of the
manuscript we either assume that the data points $x_j$ are drawn
independently from a given distribution $\MeaPr$, or from the Markov
chain $x_{j+1} = \FlowXDiscrete{t}{x_j}$, started at a known point
$x_1\in\MeaSp$. Some convergence properties are guaranteed if the
Markov chain is Ergodic~\cite{korda2018convergence}.

The EDMD procedure for characterising the dynamical system that
generated the data starts with a choice of linearly independent
functions $\Basis_j:\MeaSp\to\TargetField$, $j=1,\dots,\LimBasis$.  By
first defining the vector field
$\BBasis : \MeaSp \to \TargetField^\LimBasis$ as
\begin{equation}
  \BBasis(x) =  \vecbracks{\Basis_1(x),\dots,\Basis_\LimBasis(x)}^\intercal  ,
\end{equation}
we specify the matrices
$\BBasis(\Xv), \BBasis(\Yv) \in \TargetField^{\LimBasis \times
  \LimDat}$ as
\begin{equation}
  \BBasis(\Xv)=\vecbracks{\BBasis(x_1),\dots,\BBasis(x_\LimDat)}  ,
  \quad
  \BBasis(\Yv)=\vecbracks{\BBasis(y_1),\dots,\BBasis(y_\LimDat)}  .
\end{equation}
Given the basis functions and the data, we solve the linear least
squares problem
\begin{equation}\label{eq:edmd_mat_min}
  \min_{A\in\mathbb{R}^{\LimBasis\times\LimBasis}}\norm{A\BBasis(\Xv)-\BBasis(\Yv)}_F^2
  =\min_{A\in\mathbb{R}^{\LimBasis\times\LimBasis}}\sum_{j=1}^\LimDat\norm{A\BBasis(x_j)-\BBasis(y_j)}_2^2.
\end{equation}
The map
$\norm{A}_F^2:=\sum_{i=1}^N \sum_{j=1}^N \vert A_{i,j} \vert^2$ is
known as the Frobenius norm.  Let $A^\dagger$ denote the Moore-Penrose
pseudo-inverse of a matrix $A$~\cite{penrose1955generalized}.  A
solution to the minimisation is given by
\begin{align}\label{eq:def_edmdmat}
  \EDMDmat &= \BBasis(\Yv)\BBasis(\Xv)^\dagger \nonumber \\
           &=  \BBasis(\Yv){\BBasis(\Xv)}^\intercal {\left(
             {\BBasis(\Xv)\BBasis(\Xv)}^\intercal \right)}^{-1}  ,
\end{align}
where the second equality holds so long as the rows of $\BBasis(\Xv)$
are linearly independent.  Using the EDMD matrix, one can construct a
linear operator that approximately characterises the evolution of
functions under the dynamical system that generated the data.
\begin{definition}
  Let $\FunSpGFin$ be the linear span of the functions
  $\Basis_1,\dots,\Basis_\LimBasis$.  Define the time $t\geq 0$ EDMD
  operator $\EDMDop^t:\FunSpGFin\to\FunSpGFin$ on functions
  $g(x)=c_g^\intercal\BBasis(x)$, with $c_g\in\mathbb{R}^\LimBasis$,
  by
  \begin{equation}
    \EDMDop^t g(x) = c_g^\intercal\EDMDmat \BBasis(x).
  \end{equation}
\end{definition}

%%%%%%%%%%%%%%%%%%%%%%%%%%%%%%%%%%%%%
\subsection{Connection between Koopman and EDMD
  operators}\label{subsec:edmd_koop_conn}
%%%%%%%%%%%%%%%%%%%%%%%%%%%%%%%%%%%%%

The EDMD operator approximates the projection of the Koopman operator
onto $\FunSpGFin$, %sometimes referred to as a Galerkin approximation,
with respect to a data-driven inner product.  Assume that there exists
a subspace of $L^2(\MeaSp,\MeaPr)$ that is invariant under $\Koopd^t$
for all $t\geq 0$, and denote the largest such space by
$\FunSpGInf$. We will henceforth consider the Koopman operators
$\Koopd^t:\FunSpGInf\to\FunSpGInf$.  Assume that the basis functions
$\Basis_j$ belong to $\FunSpGInf$ for $j=1,\dots,\LimBasis$, so that
$\FunSpGFin\leq \FunSpGInf$.  The $L^2(\MeaPr)$ projection onto
$\FunSpGFin$ is defined as
\begin{equation}
  \Proj_{\BBasis}^\MeaPr  g
  % = \argmin_{f\in \FunSpGFin}
  % \norm{f-g}_{L^2(\MeaPr)}
  =\argmin_{f\in\FunSpGFin}
  \int_\MeaSp{(f-g)}^2\,\dd \MeaPr
  = \BBasis^\intercal\argmin_{c\in\TargetField^{\LimBasis}}
  \int_\MeaSp{(c^\intercal\BBasis-g)}^2\,\dd \MeaPr.
\end{equation}
As $\FunSpGFin$ is finite-dimensional, the projected Koopman operator
$\ProjKoopd{\MeaPr}^t:\FunSpGFin\to\FunSpGFin$ defined by
$\ProjKoopd{\MeaPr}^t=\Proj_{\BBasis}^\MeaPr \Koopd^t$ has an
associated matrix representation.  This matrix is can be written as
$\Koopdmat{\MeaPr}=\Transmat{\MeaPr} \Massmat{\MeaPr}^{-1}$, where
\begin{equation}\label{eqn:mass_trans_mat_def}
  \Massmat{\MeaPr} = \int_\MeaSp\BBasis\BBasis^\intercal\,\dd\MeaPr,
  \qquad
  \Transmat{\MeaPr} = \int_\MeaSp \left(\Koopd^t\BBasis\right)\BBasis^\intercal\,\dd\MeaPr.
\end{equation}
The matrix $\Massmat{\MeaPr}$ is often referred to as the mass matrix,
or the Gramian matrix.  The next result and subsequent proof is
adapted from~\cite{korda2018convergence}.
\begin{theorem}\label{thm:proj_koop_rep}
  If the matrix $\Massmat{\MeaPr}$ is invertible, then for
  $g=\Transpose{c_g}\BBasis\in\FunSpGFin$ we have
  \begin{equation}
    \ProjKoopd{\MeaPr}^t g = \Transpose{c_g} \Koopdmat{\MeaPr} \BBasis.
  \end{equation}
\end{theorem}
\begin{proof}
  By the definition of the projection operator,
  \begin{align}
    \Proj_{\BBasis}^\MeaPr \Koopd^tg
    &= \Transpose{\BBasis}\argmin_{c\in\TargetField^\LimBasis}
      \int_\MeaPr{[\Transpose{c}\BBasis-
      \Transpose{c_g}\left( \Koopd^t\BBasis \right)]}^2\,\dd\MeaPr\\
    &= \Transpose{\BBasis}\argmin_{c\in\TargetField^\LimBasis}
      \left[ \Transpose{c}\Massmat{\MeaPr} c -
      2\Transpose{c}\Transpose{\Transmat{\MeaPr}}c_g \right].
  \end{align}
  The minimiser is unique, and is given by
  $c=\Massmat{\MeaPr}^{-1}\Transpose{\Transmat{\MeaPr}}c_g$.  It
  therefore follows that
  \begin{equation}
    \ProjKoopd{\MeaPr}^t g
    = \Transpose{c}\BBasis=
    \Transpose{c_g}\Transmat{\MeaPr}\Massmat{\MeaPr}^{-1}\BBasis
    =\Transpose{c_g}\Koopdmat{\MeaPr}\BBasis.
  \end{equation}
\end{proof}
For the rest of the article, we assume that $\Massmat{\MeaPr}$ is
invertible.

Define the data-driven measure $\MeaPrDat$ from the input data, so
that
\begin{equation}\label{eq:def_meapr_dat}
  \MeaPrDat(x) = \frac{1}{\LimDat}\sum_{j=1}^\LimDat \delta(x-x_j),
\end{equation}
where $\delta(x)$ is the Dirac delta function.  For data coming from a
deterministic dynamical system, such as an ODE, we show that the
EDMD-operator is equal to the $L^2(\MeaPrDat)$-projection of
$\Koopd^t$. Further, we provide evidence for why the EDMD operator is
also a reasonable estimator in the SDE case.

\begin{proposition}\label{prop:edmd_koop_equal}
  If the system is deterministic, that is, $\Vol(x;\vParam)\equiv 0$, then
  the projected Koopman operator
  $\ProjKoopdDat^t=\Proj_{\BBasis}^{\MeaPrDat }\Koopd^t$ is equal to
  the EDMD operator $\EDMDop$:
  \begin{equation}
    \ProjKoopdDat^tg = \EDMDop g,
    \qquad \forall g\in \FunSpGFin.
  \end{equation}
\end{proposition}
\begin{proof}
  We prove the equivalence of the finite-dimensional, linear operators
  by showing that their matrix representations are equal.  Remember
  that the EDMD matrix is given by
  \begin{equation}
    \EDMDmat = \BBasis(\Yv){\BBasis(\Xv)}^\intercal{\left(
        {\BBasis(\Xv)\BBasis(\Xv)}^\intercal \right)}^{-1}.  \nonumber
  \end{equation}
  Under the empirical measure $\MeaPrDat$, we have that
  \begin{equation}
    \Massmat{\MeaPrDat}=
    \int_\MeaSp\BBasis(x)\Transpose{\BBasis(x)}\,\dd\MeaPrDat(x)
    =\frac{1}{\LimDat} \left( \BBasis(\Xv)\Transpose{\BBasis(\Xv)} \right).
  \end{equation}
  Also, note that for data pairs $(x_j,y_j)$ coming from a
  deterministic system,
  \begin{equation}
    \Koopd^tg(x_j)= \mathbb E_W[g(\FlowX{x_j})] = g(y_j).
  \end{equation}
  In the same fashion as for the mass matrix, the empirical measure
  then implies that
  \begin{equation}
    \TransmatDat =
    \frac{1}{\LimDat}\left(
      (\Koopd^t\BBasis(\Xv))\Transpose{\BBasis(\Xv)}
    \right)
    =\frac{1}{\LimDat}\left(
      \BBasis(\Yv)\Transpose{\BBasis(\Xv)}
    \right).
  \end{equation}
  Thus, the result follows since
  $\KoopdmatDat=\TransmatDat \Massmat{\MeaPrDat}^{-1} =
  \EDMDmat$.
\end{proof}

The proof of \Cref{prop:edmd_koop_equal} indicates that the EDMD
operator approximates the projected Koopman operator for an SDE.\@ The
approximation error arises when approximating $\TransmatDat$ with
$\EDMDTransmat$, by replacing
$\Koopd^t\Basis(x_j)=\mathbb{E}_W\left[ \Basis(\FlowX{x_j}) \right]$
with $\Basis(y_j)$.  The proposed algorithm in the next section is
based on matching the Frobenius norm of the matrix representation of
operators, and we shall thus discuss how well the EDMD matrix
approximates the projected Koopman operator matrix.  To start, we show
that the matrix $\EDMDTransmat$ is an unbiased estimator of
$\TransmatDat$ with variance of the order $\mathcal{O}(T^{-1})$. To
show this we must first make some assumptions on the second-order
moments arising from the SDE and the sampling of the $x_j$.

\begin{myassumption}\label{ass:data_bounds}
  For a fixed $t>0$, let the data ${(x_j)}_{j=1}^T$ be drawn either
  independently from a probability measure $\mu$, or from
  a Markov chain $x_{j+1} = \FlowXDiscrete{t}{x_j}$, started at a
  known point $x_1\in\MeaSp$.
  Let $f,g\in
  \{\Basis_1,\dots,\Basis_\LimBasis\}$.
  Assume that there exists a $\tilde{\gamma}>0$ that satisfies the
  following.

  In the case when the $x_j$ are i.i.d.~drawn from $\mu$,
  \begin{equation}\label{eq:independent_second_moment_bound}
    \mathbb{E}_x\left[{g(x)}^2\mbox{Var}_W[f(\FlowXDiscrete{t}{x})]\right]
    \leq \tilde{\gamma}.
  \end{equation}
  In the case where the $x_j = \FlowXDiscrete{jt}{x_1}$, then for any
  $l\geq 1$ we have
  \begin{equation}\label{eq:markov_second_moment_bound}
    \mathbb{E}_W\left[{g(\FlowXDiscrete{lt}{x_1})}^2\mbox{Var}[f(\FlowXDiscrete{{(l+1)t}}{x_1})]\right]
    \leq \tilde{\gamma}.
  \end{equation}
\end{myassumption}

If the basis functions are bounded,
then~\eqref{eq:independent_second_moment_bound} holds automatically.
Note that~\eqref{eq:markov_second_moment_bound} will be satisfied for
an ergodic Markov chain with a stationary distribution $\mu$
for which~\eqref{eq:independent_second_moment_bound} holds.
In our numerical examples the basis functions are all bounded on
$\MeaSp$.

\begin{proposition}\label{prop:kmatrix_estimator}
  Fix $t>0$ and the basis $\{\Basis_1,\dots,\Basis_\LimBasis\}$, and
  let the assumptions of \Cref{ass:data_bounds} hold.
  Define $\TransmatDat =
  \frac{1}{\LimDat}\left(
    \Koopd^t\BBasis(\Xv)\right)\Transpose{{\BBasis(\Xv)}}$ and
  $\EDMDTransmat=\frac{1}{\LimDat}\BBasis(\Yv)\Transpose{{\BBasis(\Xv)}}$.

  Then, by taking expectations over the distributions of $\Xv$ and $\Yv$,
  \begin{equation}\label{eq:transmat_unbiased_estimator}
    \mathbb{E}_{\Xv}[\TransmatDat] = \mathbb{E}_{\Xv,\Yv}[\EDMDTransmat].
  \end{equation}
  Further, there exists a $\gamma>0$ independent of $T$,
  such that
  \begin{equation}\label{eq:transmat_estimator_variance_bound}
    \mathbb{E}_{\Xv,\Yv}\norm{\TransmatDat-\EDMDTransmat}_F^2
    \leq \gamma T^{-1}.
  \end{equation}
\end{proposition}

\begin{proof}
  We start by showing that~\eqref{eq:transmat_unbiased_estimator}
  holds. First, note that the ${(i,j)}^{\mathrm{th}}$ elements are given by
  \begin{align}
    {(\EDMDTransmat)}_{i,j}
    &= \frac{1}{T}\sum_{k=1}^T
      \Basis_i(y_k)\Basis(x_k)\\
    {(\TransmatDat)}_{i,j}
    &= \frac{1}{T}\sum_{k=1}^T
      \Koopd^t\Basis_i(x_k)\Basis_j(x_k).
  \end{align}
  Let $f,g \in \{\Basis_1,\dots,\Basis_\LimBasis\}$, and
  consider the following expectation with respect to the distribution of
  $(\Xv,\Yv)$.
  \begin{align}
    \mathbb{E}_{\Xv,\Yv}\left[ \sum_{k=1}^T f(y_k)g(x_k) \right]
    &= \sum_{k=1}^T \mathbb{E}_{x_k,y_k} \left[ f(y_k)g(x_k) \right]\\
    &= \sum_{k=1}^T
      \mathbb{E}_{x_k} \left[ \mathbb{E}_W[f(\FlowXDiscrete{t}{x^k})]g(x_k)  \right]\\
    &= \sum_{k=1}^T \mathbb{E}_{x_k} \left[ \Koopd^t f(x_k)g(x_k)  \right]\\
    &= \mathbb{E}_{\Xv} \left[\sum_{k=1}^T  \Koopd^t f(x_k)g(x_k)  \right].
  \end{align}
  The second line follows from the law of total expectations, and the
  third from the definition of the Koopman operator.
  It follows that $\mathbb{E}_{\Xv,\Yv} {(\EDMDTransmat)}_{i,j} =
  \mathbb{E}_{\Xv} {(\TransmatDat)}_{i,j}$ for
  $i,j=1,\dots,\LimBasis$, which
  proves~\eqref{eq:transmat_unbiased_estimator}.

  To prove the mean-square bound, we again consider
  $f,g\in\{\Basis_1,\dots,\Basis_\LimBasis\}$.
  By linearity
  \begin{align}\label{eq:estimator_variance}
    \mathbb{E}_{\Xv,\Yv}{\left[
    \sum_{k=1}^T  g(x_k)\{\Koopd^t f(x_k)-f(y_k)\}
    \right]}^2
    =\sum_{k=1}^T\mathbb{E}_{x_k,y_k}
    \left[ {g(x_k)}^2 {\{\Koopd^t f(x_k)-f(y_k)\}}^2 \right]\\
    + 2\sum_{l<k} \mathbb{E}_{x_k,x_l,y_k,y_l}
    \left[ g(x_k)g(x_l)\{\Koopd^t f(x_k)-f(y_k)\}\{\Koopd^t f(x_l)-f(y_l)\} \right].
  \end{align}
  All the terms in the second sum are zero. To see this, first note that,
  by the law of total expectations, each term in the sum are equal to
  \begin{equation}\label{eq:zero_covariance}
    \mathbb{E}_{x_k,x_l,y_k}
    \left[ g(x_k)g(x_l)\{\Koopd^t f(x_k)-f(y_k)\}\{\Koopd^t f(x_l)-\mathbb{E}_{y_l}[f(y_l)]\} \right].
  \end{equation}
  Then, because $\mathbb{E}_{y_l}[f(y_l)] =
  \mathbb{E}_W[f(\FlowXDiscrete{t}{x_l})] = \Koopd^t f(x_l)$, it
  follows that the final bracket in~\eqref{eq:zero_covariance} is zero.
  In a similar fashion we see that the terms in the
  first sum of~\eqref{eq:estimator_variance} equal
  \begin{equation}\label{eq:fg_variance_equivalent}
    \mathbb{E}_{x_k}\left[ {g(x_k)}^2\
      \mathbb{E}_{y_k}{\Bigl( \mathbb{E}_{y_k}[f(y_k)]  -
        f(y_k)\Bigr)}^2 \right]
    =
    \mathbb{E}_{x_k}\left[ {g(x_k)}^2\
      \mbox{Var}_{y_k}[f(\FlowXDiscrete{t}{x_k})] \right].
  \end{equation}
  By \Cref{ass:data_bounds} this expectation is bounded by $\tilde{\gamma}$ for $k=1,\dots,T$.
  We therefore have
  \begin{equation}
    \mathbb{E}_{\Xv,\Yv}{\left[
        \sum_{k=1}^T  g(x_k)\{\Koopd^t f(x_k)-f(y_k)\}
      \right]}^2 \leq \tilde{\gamma}T.
  \end{equation}
  The result therefore follows by summing over all the $\LimBasis^2$
  entries of the matrices,
  \begin{equation}
    \mathbb{E}_{\Xv,\Yv}\norm{\TransmatDat-\EDMDTransmat}_F^2
    \leq \frac{\LimBasis^2}{T^2} \tilde{\gamma} T
    = \gamma T^{-1},
  \end{equation}
  where $\gamma =  \LimBasis^2\tilde{\gamma}$.
\end{proof}

Using \Cref{prop:kmatrix_estimator} we finally state a result on how
well the matrix representation of the EDMD operator approximates the Koopman operator.

\begin{corollary}\label{corr:matrix_estimator_bound}
  Let the assumptions of \Cref{prop:kmatrix_estimator} hold with bound constant
  $\gamma$, and
  further assume that the second moment of the mass matrix Frobenius
  norm is bounded above, that is
  \begin{equation}
    \mathbb{E}_{\Xv}\norm{\MassmatDat^{-1}}_F^2 \leq \tilde{M},
  \end{equation}
  for some $\tilde{M}>0$, independent of $T$.

  Then the difference between the matrix representations of the EDMD
  and Koopman operators
  Koopman operator in the Frobenius norm is of order
  $T^{-1/2}$,
  \begin{equation}\label{eq:assumption_massmat_second_moment}
    \mathbb{E}_{\Xv,\Yv}\norm{\KoopdmatDat-\EDMDmat}_F
    \leq \sqrt{\frac{\gamma \tilde{M}}{T}}.
  \end{equation}
\end{corollary}

\begin{proof}
  From \Cref{prop:edmd_koop_equal,prop:kmatrix_estimator} and their proofs
  we know that the matrix representations of the projected Koopman and
  EDMD operators are $\KoopdmatDat=\TransmatDat\MassmatDat^{-1}$ and
  $\EDMDmat=\EDMDTransmat\MassmatDat^{-1}$ respectively.
  Further, the Frobenius norm of a product is bounded by the product
  of the norms,
  \begin{equation}
    \norm{(\TransmatDat-\EDMDTransmat)\MassmatDat^{-1}}
    \leq \norm{{\TransmatDat-\EDMDTransmat}}_F\norm{\MassmatDat^{-1}}_F.
  \end{equation}
  By the Cauchy-Schwarz inequality,
  \begin{align}
    \mathbb{E}_{\Xv,\Yv}\left[ \norm{{\TransmatDat-\EDMDTransmat}}_F\norm{\MassmatDat^{-1}}_F \right]
    \leq
    \sqrt{\mathbb{E}_{\Xv,\Yv}\norm{\TransmatDat-\EDMDTransmat}_F^2}
    \sqrt{\mathbb{E}_{\Xv,\Yv}\norm{\MassmatDat^{-1}}_F^2}.
  \end{align}
  The result now follows
  by applying~\eqref{eq:transmat_estimator_variance_bound}
  and~\eqref{eq:assumption_massmat_second_moment} to the right hand side.
\end{proof}

For completeness we note that the EDMD method can be adjusted to
approximate the Perron--Frobenius operator by the matrix
\cite{klus2016On}
\begin{equation}\label{eq:PFO_mat_rep}
  \Transpose{\TransmatDat} \Massmat{\MeaPrDat}^{-1}\, .
\end{equation}

%%%%%%%%%%%%%%%%%%%%%%%%%%%%%%%%%%%%%
\subsection{Parameter estimation using projected Koopman
  operators}\label{subsec:param_projkoop}
%%%%%%%%%%%%%%%%%%%%%%%%%%%%%%%%%%%%%

In the following, we write $\Koopd^t(\vParam)$ when we want to
emphasise the $\vParam$-dependence of the time-$t$ Koopman operator.
We will emphasise the $\vParam$-dependence in a similar fashion for
the projected Koopman operators and their matrix representations, when
needed.

Assume the output data $\Yv$ is generated from the SDE with a
particular parameter $\vParam^*$ and initial conditions $\Xv$.  Then
$\ProjKoopdDat^t(\vParam^* )\approx \EDMDop$, with equality whenever
$\Vol(x;\vParam^* )\equiv 0$.  Further,
\Cref{corr:matrix_estimator_bound} motivates estimating $\vParam^*$
by solving the minimisation problem
\begin{equation}\label{eq:min_koop_edmd}
  \min_{\vParam\in\ParamSet}\norm{\KoopdmatDat(\vParam)-\EDMDmat}_F^2.
\end{equation}
We choose to minimise the Frobenius norm instead of the matrix norm
induced by the inner product of $\FunSpGFin$, because it is cheaper to
calculate, and numerical investigations indicated similar
performance. Further discussion on the formulation using matrix norms
is given in \Cref{sec:variations_algo}.

Calculating the matrix $\TransmatDat(\vParam)$ in
$\KoopdmatDat(\vParam)=\TransmatDat(\vParam) \MassmatDat^{-1}$
requires the solution of the SDE, which in most cases would make this
method intractable.  We can, however, take advantage of the
infinitesimal generator of the SDE to calculate
$\KoopdmatDat(\vParam)$ cheaply.  In the remainder of the section, we
define the infinitesimal generator, explain how it can be used to
calculate $\KoopdmatDat(\vParam)$, and summarise the parameter
estimation method in \Cref{alg:edmdparam}.

\begin{definition}
  The \emph{continuous-time Koopman operator}, $\Koopd$, is the
  infinitesimal generator of the time-$t$ Koopman operators,
  \begin{equation}\label{eq:inf_gen_def}
    \Koopd g(x) = \lim_{t\downarrow 0}\frac{\Koopd^tg(x)-g(x)}{t}.
  \end{equation}
\end{definition}
It is a well-known result that the continuous-time Koopman operator is
a linear, second-order differential operator on well-behaved
functions.  See, for example,~\cite[Ch.~17]{cohen2015stochastic} for a
discussion and proof.
\begin{theorem}\label{thm:inf_gen_c2}
  Let $g\in C^2(\MeaSp)$, then $\Koopd(\vParam) = \DiffGen(\vParam)$,
  where
  \begin{equation}\label{eq:inf_gen_c2}
    \DiffGen(\vParam) g(x) =  \DiffGen \left[ g(x) ; \vParam \right]  := \Drift(x;\vParam)\frac{\dd g(x)}{\dd x} + \frac{[\Vol(x;\vParam)]^2}{2} \frac{\dd^2 g(x)}{\dd x^2}.
  \end{equation}
\end{theorem}

In the remainder of the section, we assume that the basis functions
$\Basis_j$ are sufficiently smooth for~\eqref{eq:inf_gen_c2} to hold
on $\FunSpGFin$. In addition, we require that $\FunSpGFin$ is
invariant under $\DiffGen(\vParam)$, that is,
$\DiffGen(\vParam) g\in\FunSpGFin$ for all $g\in \FunSpGFin$.  The
invariance assumption puts smoothness constraints on
$\Drift(x;\vParam)$ and $\Vol(x;\vParam)$.
\begin{theorem}\label{eq:proj_koop_mat_exp}
  Assume that $\Koopd=\DiffGen$ on $\FunSpGFin$, and that $\FunSpGFin$
  is invariant under $\DiffGen$.  Then $\Koopd^t=\mathrm{e}^{t\Koopd}$
  on $\FunSpGFin$, and the matrix representation of
  $\ProjKoopd{\MeaPr}^t$ is given by
  \begin{equation}\label{eq_exp_relation}
    \Koopdmat{\MeaPr}= \exp\left( {t\Lmat{\MeaPr}\Massmat{\MeaPr}^{-1}} \right),
  \end{equation}
  where
  \begin{equation}
    \Lmat{\MeaPr} = \int_\MeaSp \left(\DiffGen\BBasis\right)\Transpose{\BBasis}\,\dd\MeaPr.
  \end{equation}
\end{theorem}
\begin{proof}
  The equality $\Koopd^t=\mathrm{e}^{t\Koopd}$ on $\FunSpGFin$ follows
  from the definition of $\Koopd$ and the fact that linear operators
  are bounded on finite-dimensional spaces.  This also means that this
  exponentiation relation holds for the matrix representations of
  $\Koopd^t$ and $\Koopd$, with respect to a given inner product.

  Following the same argument as in the proof of
  \Cref{thm:proj_koop_rep}, we can show that
  \begin{equation}
    \Proj_{\BBasis}^\MeaPr \DiffGen g
    = \Transpose{c_g}\Lmat{\MeaPr}\Massmat{\MeaPr}^{-1}\BBasis,
    \qquad
    \forall g=\Transpose{c_g}\BBasis\in \FunSpGFin.
  \end{equation}
  Since
  $\Proj_{\BBasis}^\MeaPr \Koopd = \Proj_{\BBasis}^\MeaPr \DiffGen$ on
  $\FunSpGFin$, their matrix representations are the same.  It follows
  that the matrix representation $\Koopdmat{\MeaPr}$ of
  $\ProjKoopd{\MeaPr}^t$ is given by equation~\eqref{eq_exp_relation}.
\end{proof}

%%%%%%%%%%%%%%%%%%%%%%%%%%%%%%%%%%%%%
\subsection{Proposed Algorithm}\label{sec:chosen_algo}
%%%%%%%%%%%%%%%%%%%%%%%%%%%%%%%%%%%%%

\Cref{alg:edmdparam} describes our SDE parameter estimation method,
based on the minimisation problem~\eqref{eq:min_koop_edmd}, and
\Cref{eq:proj_koop_mat_exp}.  With a large number of basis functions,
or when extending the method to higher-dimensional state spaces,
calculating the projected Koopman and EDMD matrices may become
expensive. Computationally efficient implementations of (extended)
dynamic mode decomposition, based on SVD factorisations, can be used
to alleviate such issues. See, for example, Tu
\emph{et. al.}~\cite{tu2014dynamic} for an overview.  In the numerical
example of \Cref{sec:ou_proc}, however, we see that the algorithm
performs as well as existing SDE parameterisation methods already with
three basis functions.

\begin{algorithm}
  \caption{EDMD parameter estimation}
  \label{alg:edmdparam}
  \begin{algorithmic}[1]
    \REQUIRE{Data $\Xv = \vecbracks{x_1,\dots,x_\LimDat}$,
      $\Yv=\vecbracks{y_1,\dots,y_\LimDat}$, and time-step $t>0$}
    \REQUIRE{Basis functions
      $\BBasis(x)=\vecbracks{\Basis_1(x),\dots,\Basis_\LimBasis(x)}^\intercal$}
    \REQUIRE{Infinitesimal generator
      $\DiffGen(\vParam) = \Drift(x;\vParam)\frac{\dd }{\dd x} +
      {\textstyle \frac{1}{2}}{\Vol(x;\vParam)}^2 \frac{\dd^2 }{\dd
        x^2}$} \STATE{Set
      $M\leftarrow \frac{1}{\LimDat}
      \BBasis(\Xv)\Transpose{\BBasis(\Xv)}$} \STATE{Set
      $A \leftarrow \frac{1}{\LimDat}
      \BBasis(\Yv)\Transpose{\BBasis(\Xv)}M^{-1}$} \STATE{Define
      $\MeaPr(x) = \frac{1}{\LimDat}\sum_{j=1}^\LimDat \delta(x-x_j)$}
    \STATE{Define
      $L(\vParam) = \int_\MeaSp
      \left(\DiffGen(\vParam)\BBasis\right)\Transpose{\BBasis}\,\dd\MeaPr$}
    \STATE{Solve
      \begin{equation*}
        \hat{\vParam} =
        \argmin_{\vParam\in\ParamSet}\norm{\,\exp\left(
            {tL(\vParam)M^{-1}} \right)-A}_F^2.
      \end{equation*}
      \RETURN $\hat{\vParam}$ }
  \end{algorithmic}
\end{algorithm}
\begin{myremark}
  If $\DiffGen(\vParam)$ is linear in the parameters, one can
  pre-calculate the integrals of the matrix $L(\vParam)$ in
  \Cref{alg:edmdparam} such that each function evaluation of the
  minimisation problem is reduced to scalar-matrix and matrix
  exponentiation operations. We take advantage of this for our
  examples in \Cref{sec:num_examples}.
\end{myremark}
\begin{myremark}
  The theoretical motivation for the algorithm assumed that
  $\FunSpGFin$ was invariant under $\DiffGen$. Our choices of basis
  functions in the numerical examples do not necessarily satisfy this
  assumption.
\end{myremark}
\begin{myremark}
  Instead of matching the exponential matrix $\exp\left(
    {tL(\vParam)M^{-1}} \right)$ to the EDMD matrix $A$ in step five,
  one can match $L(\vParam)M^{-1}$ to a branch of the matrix
  logarithm $\log(A)/t$.
  As noted by~\cite{mauroy2016linear,mauroy2017koopman}, it is not clear which
  branch of the matrix logarithm to choose if we follow this approach.
\end{myremark}

%%%%%%%%%%%%%%%%%%%%%%%%%%%%%%%%%%%%%
\section{Numerical Examples}\label{sec:num_examples}
%%%%%%%%%%%%%%%%%%%%%%%%%%%%%%%%%%%%%

%%%%%%%%%%%%%%%%%%%%%%%%%%%%%%%%%%%%%
\subsection{The Ornstein--Uhlenbeck Process}\label{sec:ou_proc}
%%%%%%%%%%%%%%%%%%%%%%%%%%%%%%%%%%%%%

In this section we compare the performance of the proposed EDMD-based
parameter estimation algorithm to existing methods.  The numerical
example and data is taken from Hurn
\emph{et. al.}~\cite{hurn2007seeing}, where the authors compare the
performance of \num{14} different SDE parameter estimation algorithms.
The data from the comparison paper is available at
\url{http://www.ncer.edu.au/resources/data-and-code.php}.  The
Ornstein-Uhlenbeck equation is the SDE
\begin{equation}\label{eq:ou_sde}
  \dd \rvX= \Param_1(\Param_2-\rvX)\, \dd t +  \Param_3\dd W_t \, .
\end{equation}
The infinitesimal generator of solutions to this SDE is
\begin{equation}
  \DiffGen(\vParam)
  = \Param_1(\Param_2-x)\frac{\dd}{\dd x} +
  \frac{\Param_3^2}{2}\frac{\dd^2}{\dd x^2}.
\end{equation}
Note that evaluation of the matrix $\Lmat{\MeaPr}(\vParam)$ can be
very cheap, by pre-calculating the integrals involved:
\begin{align}
  \Lmat{\MeaPr}(\vParam)
  = \Param_1\Param_2\int_\MeaSp \frac{\dd \BBasis(x)}{\dd
  x}\Transpose{\BBasis(x)}\,\dd \MeaPr(x)
  &- \Param_1 \int_\MeaSp x\frac{\dd \BBasis(x)}{\dd
    x}\Transpose{\BBasis(x)}\,\dd \MeaPr(x)\\
  &+ \frac{\Param_3^2}{2}\int_{\MeaSp}\frac{\dd^2 \BBasis(x)}{\dd
    x^2}\Transpose{\BBasis(x)}\,\dd \MeaPr(x).\nonumber
\end{align}
Then, subsequent evaluations of $\Lmat{\MeaPr}(\vParam)$ are simply
scalar-matrix computations.

The data set from~\cite{hurn2007seeing} consists of \num{2000}
independent sample paths with \num{501} data points each, separated
with a time step $\Delta t = 1/12$. They are all drawn from the SDE
with parameter $\vParam^* =(0.2,0.08,0.03)$.  For each sample path, we
employ \Cref{alg:edmdparam} to estimate $\vParam^* $.  For this
example, we test the performance of radial basis functions (RBFs) of
the form
\begin{equation}\label{eq_Gauss_RBF}
  \Basis_j(x) = \exp\left\{ -l^2{( x - c_j )}^2 \right\},\qquad j=1,\dots,\LimBasis,
\end{equation}
where $l>0$ is a given length scale, and $c_j$ an increasing
sequence of centre points.  With $\LimBasis$ basis functions, the
centre points are chosen to be spaced at a distance
$\Delta x_\LimBasis>0$ apart, so that
$c_1 = \min\{\Xv\}+\Delta x_\LimBasis$ and
$c_N= \max\{\Xv\}-\Delta x_\LimBasis$. The length scale is set
to $l=1/\Delta x_\LimBasis$. The parameter estimation is done with
$N=3,4,5$. In \Cref{fig:rbfs_3}, the basis functions, when $N=3$, are
shown for the first sample path in the data set.
\begin{figure}[htb]
  \centering
  \includegraphics[width=0.85\textwidth,axisratio=1.62]{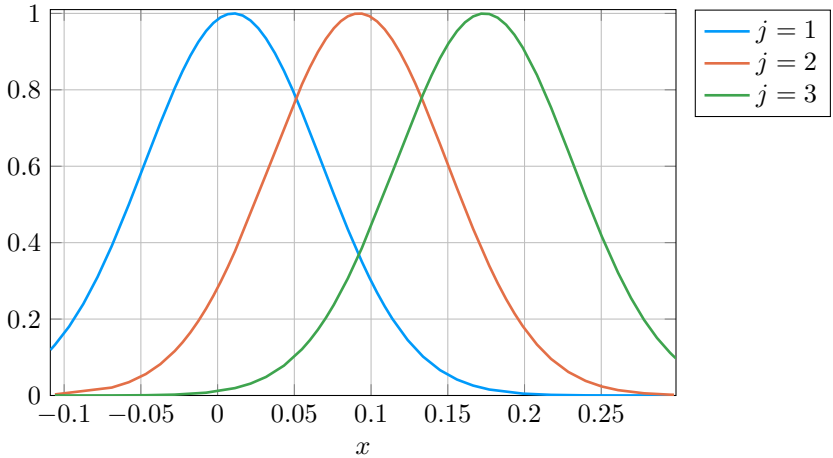}
  \caption{The RBFs $\Basis_j(x)$ for $N=3$, with centre points and
    length scale calculated based on the first sample path in the data
    set.}\label{fig:rbfs_3}
\end{figure}
\begin{myremark}
  The basis functions depend on the data, which is not consistent with the
  assumptions of
  \Cref{prop:kmatrix_estimator,corr:matrix_estimator_bound}.
  We note, however, that as $\LimDat$ increases, the probability that
  $\max\{\Xv\}$ and $\min\{\Xv\}$ varies becomes small.
  One can also restrict the $c_j$ to be within a box, so that
  eventually the basis functions stay fixed.
\end{myremark}

The minimisation step uses the BFGS
algorithm~\cite{nocedal2006numerical}, with interpolation
backtracking, implemented by the \texttt{Optim.jl}
package~\cite{mogensen2018optim} in the Julia programming
language~\cite{bezanson2017julia}.  Derivatives are calculated using
finite differences.  The initial guess is set to $\vParam^* $, to
be consistent with the comparisons done in Hurn
\emph{et. al.}~\cite{hurn2007seeing}.
Numerical investigations, not included here, show that the objective function
is convex for the different data sets.

In \Cref{tbl:ou_comparison}, statistics of the performance of the
algorithm for $N=3,4,5$ is presented.  For a few of the \num{2000}
sample paths, the algorithm estimation is far off the correct
parameter $\vParam^* $, reporting $|\hat{\Param}_2|$ to be orders of
magnitude too large. These are typically sample paths for which the
data predominantly stays on one side of the long-term mean $\Param_2^*$, so that
the mean-reversion property of the process is less apparent.
We exclude these paths from the calculation of
the statistics, and instead report any values where
$|\hat{\Param}_j |>1$, for at least one of $j=1,2,3$, as failures.
Let $\vParam^{k,*}$ denote the reported parameter for the
$k^{\text{th}}$ sample path.  The table shows the bias and root mean
squared (RMS) values for the estimates, which are calculated as
$\frac{1}{\LimDat}\sum_{k=1}^{2000} (\Param_j^{k,*}-\Param_j^* )$ and
$\sqrt{\frac{1}{\LimDat}\sum_{k=1}^{2000} {(\Param_j^{k,*}-\Param_j^*
    )}^2}$ respectively.  These values are compared to the results
from the exact maximum likelihood (EML) reference algorithm used in
Hurn \emph{et. al.}~\cite{hurn2007seeing}. Note that parameter
estimation with EML is only available if one knows the transition
probability density of the associated SDE.  The other \num{13}
algorithms reported very similar performance statistics. From the
table, we can see that our proposed algorithm performs just as well as
the reference EML algorithm.
\begin{table}[htb]
  \centering
  \begin{tabular}{l*{3}{ S[table-auto-round,table-format=-1.4]
    S[table-auto-round,table-format=-1.4]}r}
    \toprule
    Alg&\multicolumn{2}{c}{$\Param_1$}&\multicolumn{2}{c}{$\Param_2$}&\multicolumn{2}{c}{$\Param_3$}&Fail\\
    \cmidrule(lr){2-3}\cmidrule(lr){4-5}\cmidrule(lr){6-7}
       &{Bias}&{RMSE}&{Bias}&{RMSE}&{Bias}&{RMSE}&\\
    \midrule
    RBF(3)&0.08500297960614991&0.17477197244896076&-0.0013887799878454991&0.03166719057401393&-0.0000272341&0.0023271247467603704&7\\
    RBF(4)&0.085953839248921&0.18799699313965615&-0.001996290952441892&0.042771997777942564&0.00019422483091887527&0.001877228488635663&22\\
    RBF(5)&0.0907590910354649&0.18735773850369386&-0.0006607759497238338&0.03723568550108482&0.0002664163213895104&0.0017060640461324563&19\\
    EML&0.1101&0.1780&-0.0006&0.0227&0.0001&0.0010&---\\
    \bottomrule
  \end{tabular}
  \caption{Performance statistics comparing \Cref{alg:edmdparam} with
    different number of RBFs. The row with the EML results are taken
    from~\cite{hurn2007seeing}.  }\label{tbl:ou_comparison}
\end{table}

%%%%%%%%%%%%%%%%%%%%%%%%%%%%%%%%%%%%%
\subsubsection{Performance with increasing amount of
  data}\label{subsec:ouproc_example_conv}
%%%%%%%%%%%%%%%%%%%%%%%%%%%%%%%%%%%%%

We end the Ornstein-Uhlenbeck example by investigating the estimation
improvement with increasing amount of data $\LimDat$ while keeping
the time step fixed. \Cref{corr:matrix_estimator_bound} suggests that
the EDMD matrix will better approximate the projected Koopman matrix
as $\LimDat$ increases.
To this end, we created \num{2000} sample paths of the
solution to~\eqref{eq:ou_sde}, with $\vParam^* = (0.2,0.08,0.03)$, all
started at the initial condition $X_0 = \Param_2=0.08$.  The data is
stored at time steps $\Delta t=1/12$ apart, and generated from the
exact conditional distribution given by
\begin{equation}
  \rvX^x \sim
  \mathcal N(\Param_2+(x-\Param_2)e^{-\Param_1 t},\,
  \Param_3^2\left( 1-e^{-2\Param_1 t} \right)/2\Param_1),
\end{equation}
so that $x_{j+1}=\FlowXDiscrete{\Delta t}{x_{j}}$ and $y_{j}=x_{j+1}$
for $j=1,\dots,\LimDat-1$.
\Cref{fig:ou_rmse_convergence} reports the root mean squared error of
the estimators, with data amount $\LimDat = 500\times 2^{j}$, for
$j=0,1,\dots,9$.  We use RBFs calculated in the same way as in
\Cref{sec:ou_proc}, with $\LimBasis = 3,5,10$.  The number of failures
are zero for the larger amounts of data, in particular no estimations
are considered a failure for $\LimBasis=3$ and $j\geq 1$.  The RMSE
for $\Param_1$, and $\Param_2$ decreases with data of order
$\mathcal O(\LimDat^{-1/2})$, and the RMSE for $\Param_3$ decreases
approximately as $\mathcal O(\LimDat^{-1/3})$.

\begin{figure}[htb]
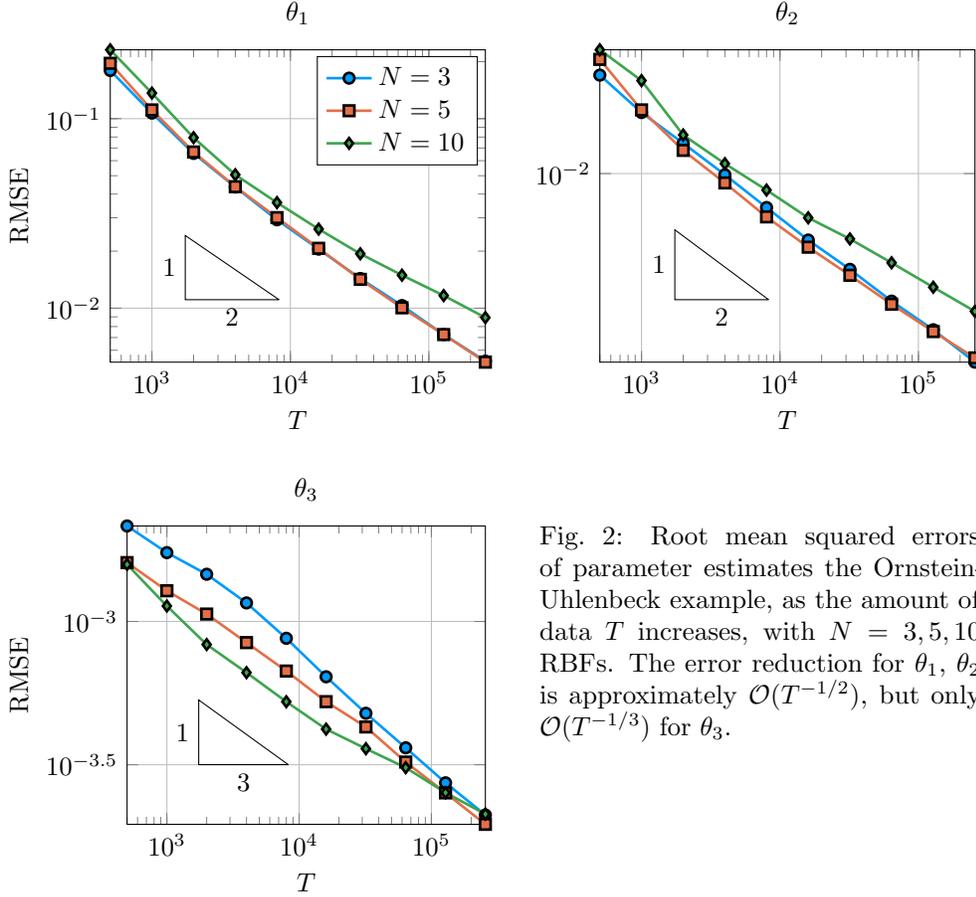

  \centering
  % \begin{subfigure}[b]{0.5\textwidth}
  \begin{minipage}[b]{0.5\textwidth}
    \vspace{0pt}
    \includegraphics[width=\textwidth,axisratio=1.2]{ou_rmse_convergence_theta1}
    % \end{subfigure}%
  \end{minipage}%
  % \begin{subfigure}[b]{0.5\textwidth}
  \begin{minipage}[b]{0.5\textwidth}
    \vspace{0pt}
    \includegraphics[width=\textwidth,axisratio=1.2]{ou_rmse_convergence_theta2}
    % \end{subfigure}%
  \end{minipage}
  % \begin{subfigure}[b]{0.5\textwidth}
  \begin{minipage}[t]{0.5\textwidth}
    \vspace{0pt}
    \includegraphics[width=\textwidth,axisratio=1.2]{ou_rmse_convergence_theta3}
    % \end{subfigure}%
  \end{minipage}
  \hfill
  \begin{minipage}[t]{0.45\textwidth}
    \vspace{1em}
    \caption{Root mean squared errors of parameter estimates the
      Ornstein-Uhlenbeck example, as the amount of data $\LimDat$
      increases, with $\LimBasis=3,5,10$ RBFs.  The error reduction
      for $\Param_1$, $\Param_2$ is approximately
      $\mathcal{O}(\LimDat^{-1/2})$, but only
      $\mathcal{O}(\LimDat^{-1/3})$ for $\Param_3$.
    }\label{fig:ou_rmse_convergence}
  \end{minipage}
\end{figure}

%%%%%%%%%%%%%%%%%%%%%%%%%%%%%%%%%%%%%
\subsection{Bounded Mean Reversion Process}\label{sec_num_example_1}
%%%%%%%%%%%%%%%%%%%%%%%%%%%%%%%%%%%%%

In \Cref{sec:ou_proc}, we saw that our algorithm performs comparably
well to other methods from the literature. Within the study of EDMD, a
common theme is the calculation of eigenfunctions (related to left
eigenvectors of the EDMD matrix) to examine the types of behaviour of
the system. For systems with many timescales, or those that are
confined to a low dimensional manifold, the eigenfunctions can be used
to offer a low dimensional description of the system.

%\todo[inline]{Reviewer comment: What is the conclusion of this
%  example? How do you choose $N$ and $J$? Would the results have been
%  the same with another value of $\Param$?}

For a matrix $A$ in the minimisation problem in equation
\eqref{eq:min_koop_edmd} with ordered eigenvalues
$1= \lambda_1 > \vert \lambda_2 \vert > \dots > \vert
\lambda_\LimBasis \vert$, and left and right eigenvectors (denoted
$\vect{w}$ and $\vect{v}$), then we can write the
$\LimEig$-eigendecomposition of $A$ as
\begin{equation}
  A_\LimEig = \sum_{j = 1}^\LimEig \lambda_j \vect{v}_j \Transpose{\vect{w}_j} \, ,
\end{equation}
and $A_\LimEig = A$ when $\LimEig = \LimBasis$. We then consider
replacing $A$ by $A_\LimEig$ in equation \eqref{eq:min_koop_edmd} for
varying $\LimEig$ and $\LimBasis$. We can interpret $\LimBasis$ as a
parameter that controls the possible resolution of the data, and
$\LimEig$ the parameter that specifies the maximum allowed resolution
in the generator.

To avoid numerical artefacts regarding sampling of data, we have
devised a numerical experiment where: we use a large amount of data
points; and we vary the types of basis functions used by considering
global basis functions, and deterministically placed RBFs (rather than
depending on the range of a sample path as in \Cref{sec:ou_proc}). We
consider the SDE with parameters $\vParam = (\Param_1,\Param_2)$
\begin{align}\label{eq_SDE_example}
  \dd X_t  = -2 \Param_1 X_t \dd t + \sqrt{2 \Param_2 (1-X_t^2)} \dd W_t \, ,
\end{align}
which is a mean reversion process (to $x=0$) bounded on the interval
$(-1,1)$. The SDE given by equation \eqref{eq_SDE_example} has
infinitesimal generator
\begin{align}
  \DiffGen(\vParam) = \Param_2 (1-x^2) \frac{\dd^2}{\dd x^2}  - 2 \Param_1 x \frac{\dd}{\dd x} \, .
\end{align}
Similar to in \Cref{sec:ou_proc}, the matrix $\Lmat{\MeaPr}(\vParam)$
can be pre-calculated. We then consider the following basis functions
\begin{itemize}
\item[1.]  Chebychev polynomials, defined on $[-1,1]$ as
  \begin{align}
    T_0(x) &= 1 \, ,  \nonumber  \\
    T_1(x ) &= x \, , \nonumber \\
    T_{n+1}(x) &= 2x T_n (x) - T_{n-1}(x) \, ,
  \end{align}
  and then $\Basis_j = T_{j-1}$ for $j=1\dots \LimBasis$.
\item[2.]  Gaussian RBFs as given by equation \eqref{eq_Gauss_RBF}. We
  position each basis equally along the interval $[-1,1]$ so
  $c_j = -1 + 2(j-1)/(\LimBasis -1)$ and choose the scaling
  constant to be the distance between points, $l = 2/(\LimBasis -1)$.
\item[3.]  Legendre polynomials, defined on $[-1,1]$ as
  \begin{align}
    P_0(x) &= 1 \, , \nonumber  \\
    P_1(x) &= x \, , \nonumber  \\
    (n+1) P_{n+1}(x) &= (2n+1)x P_n (x) - n P_{n-1}(x) \, ,
  \end{align}
  and then $\Basis_j = P_{j-1}$ for $j=1\dots \LimBasis$.
\end{itemize}
This choice was made as Chebychev polynomials are a popular choice for
polynomial basis functions on bounded intervals; RBFs offer
customisable ways of spanning a domain (with a multitude of methods
choosing the centre locations); and Legendre polynomials are the
eigenfunctions of the infinitesimal generator when $\vParam^* = (1,1)$.

In \Cref{fig_Simulation_Results} we show a comparison between the
different basis functions to estimate $\vParam^* = (1, 1)$
for different numbers of basis functions $\LimBasis$ and different
numbers of eigenfunctions in the eigen-expansion
$\LimEig \leq \LimBasis$. The simulation was set up with a very large
number of data points, we sample for 1000 time units using the
Milstein method with a time step of $\De t = 2^{-12}$. For the
Chebychev and Legendre polynomials
$2\leq \LimEig \leq \LimBasis \leq 46$, and for the RBFs
$4\leq \LimEig \leq \LimBasis \leq 92$.

From \Cref{fig_Simulation_Results}, when $\LimBasis = \LimEig$ we find
that Chebychev polynomials estimate the parameters well for all
$\LimBasis$ --- we note that in \Cref{fig_Simulation_Results}, it is
not always possible to see this effect for small $\LimBasis$. For
larger values of $\LimBasis$ we notice two trends: first, as we add
extra basis functions the estimates do not change; second, it is not
always necessary to have $\LimEig = \LimBasis$ and one can
occasionally obtain a better estimate with $\LimEig < \LimBasis$. One
interpretation of this is that with small $\LimBasis$, every
eigenfunction is important; however there is error in the data set
(generated by the Milstein SDE numerical scheme), and this error may manifest
itself in the higher order modes, so it can be beneficial to exclude
them. Finding the exact point at which to truncate is not immediately
obvious.
%\todo[inline]{Reviewer comment: What do you mean by ``there is an
%  error in the data set''? The trajectories are by nature random,
%  right?
%
%  I guess you should rewrite this to emphasise that the error is in the
%  numerical SDE solver scheme?
%}

We now consider the radial basis function results in
\Cref{fig_Simulation_Results}. One thing to note about RBFs is that
the locations and scaling parameters have change as
$\LimBasis$ varies. Therefore, one has to be careful when comparing
the system with $\LimBasis$ basis functions to the system with
$\LimBasis+1$ basis functions. This manifests itself in
\Cref{fig_Simulation_Results} as a non-monotonic behaviour for large
$\LimBasis, \LimEig$. The general trend however is that increasing the
number of basis functions improves the accuracy of the estimate, and
one should use the full eigen-expansion with $\LimEig = \LimBasis$.

Finally, the Legendre polynomial plots in
\Cref{fig_Simulation_Results} appear similarly to the Chebychev
polynomial plots. We also get the highest accuracy of parameter
estimation, however we are using \emph{a priori} information in that
we know the correct eigenfunctions.

%%%%%%%%%%%%%%%%%%%%%%%%%%%%%%%%%%%%
% \begin{wrapfigure}{R}{0.41\textwidth}
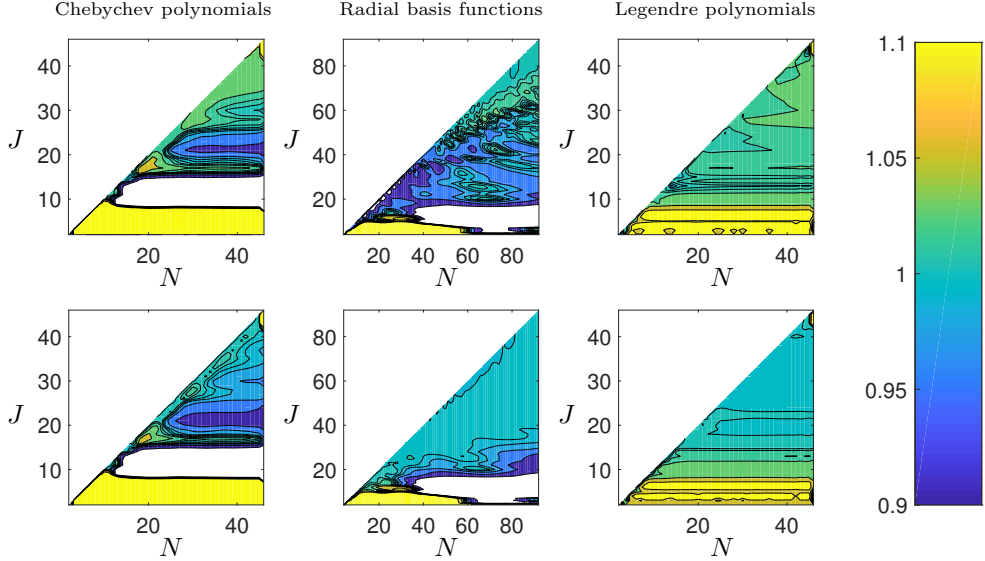
\begin{figure}[h!]
  \begin{center}
  \vspace{3em}
    \begin{overpic}[width=0.97\textwidth]{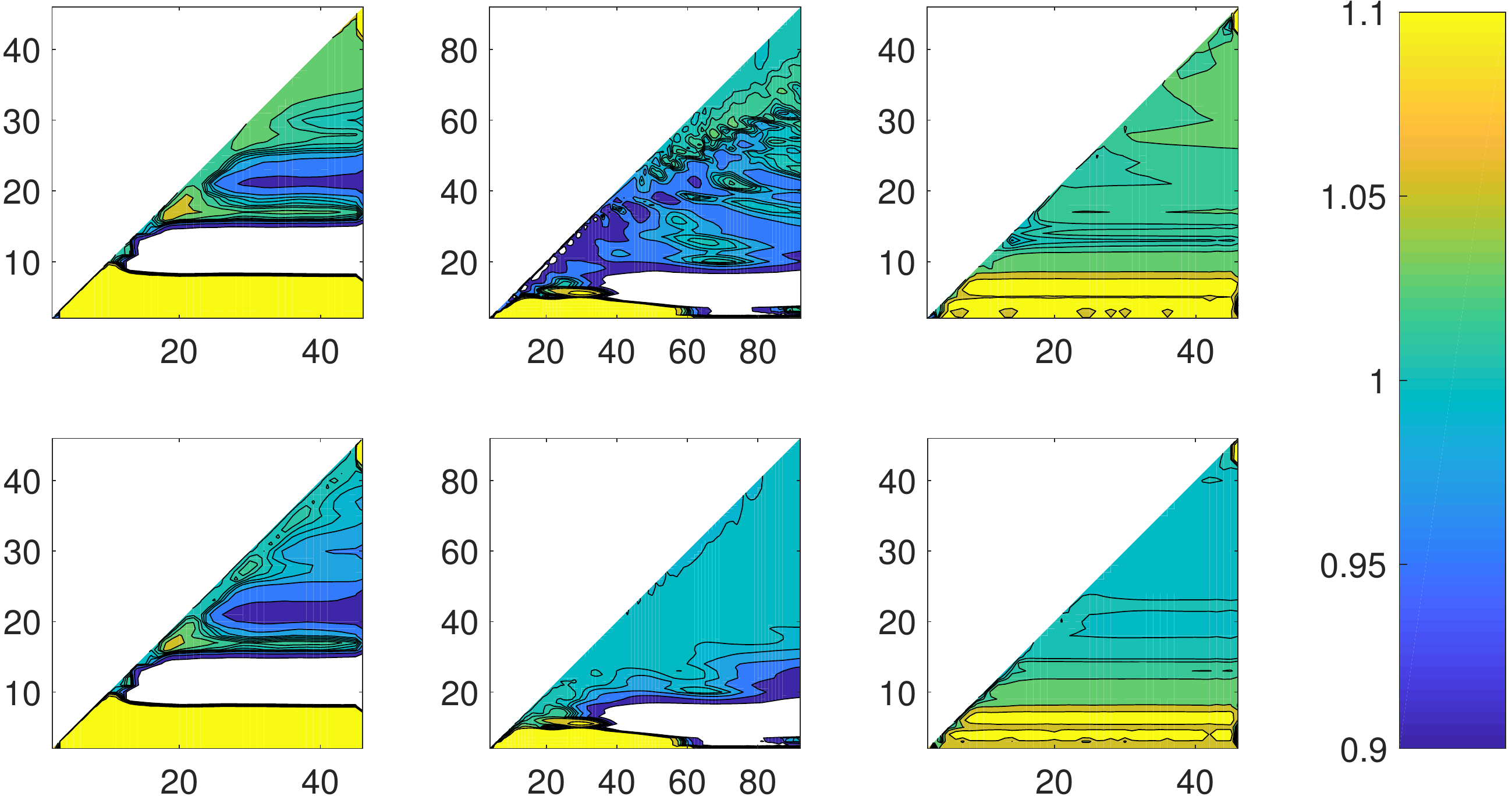}
     \put(2,55){\scriptsize Chebychev polynomials}
     \put(32,55){\scriptsize Radial basis functions}
     \put(61,55){\scriptsize Legendre polynomials}

      \put(-3,41){$\LimEig$} \put(-3,12){$\LimEig$}
      \put(26,41){$\LimEig$} \put(26,12){$\LimEig$}
      \put(55,41){$\LimEig$} \put(55,12){$\LimEig$}
      \put(13,26.5){$\LimBasis$} \put(13,-2){$\LimBasis$}
      \put(42,26.5){$\LimBasis$} \put(42,-2){$\LimBasis$}
      \put(71,26.5){$\LimBasis$} \put(71,-2){$\LimBasis$}
    \end{overpic}
  \end{center}
  \caption{ \footnotesize{ Estimate plots for parameters $\Param_1=1$
      (top row) and $\Param_2=1$ (bottom row) when varying $\LimBasis$
      and $\LimEig$. The underlying data set is generated from a
      single sample path of equation \eqref{eq_SDE_example} simulated
      using the Milstein method for 1000 time units with
      $\De t = 2^{-12}$. Estimates in the range $(0.9,1.1)$ are
      plotted in colour, estimates outside this range (including when
      $\LimEig > \LimBasis$) are in white. The contour lines are at
      $1\pm 2^{-n}/10$ for $n=0,1,2,3,4$.}
  }\label{fig_Simulation_Results}
  % \end{wrapfigure}
\end{figure}
%%%%%%%%%%%%%%%%%%%%%%%%%%%%%%%%%%%%

%%%%%%%%%%%%%%%%%%%%%%%%%%%%%%%%%%%%%
\section{Variations of the algorithm}\label{sec:variations_algo}
%%%%%%%%%%%%%%%%%%%%%%%%%%%%%%%%%%%%%
In this section we discuss variations of the algorithm with different
choices of objective. In particular, using the operator norm,
constrained EDMD, and generalised method of
moments~\cite{hansen1982large}.  In numerical tests, not all included
in this article, we find that all the methods have a (sometimes
significantly) larger function evaluation cost, without major
improvements to the parameter estimates.

In addition to the variations covered in this section, note that
operator matching of the Perron--Frobenius operator under the operator
or Frobenius norms can be done in similar ways. It follows by
considering the Perron--Frobenius matrix
$\Transpose{\TransmatDat} \Massmat{\MeaPrDat}^{-1}$
from~\eqref{eq:PFO_mat_rep}, and the similarly defined data-driven
matrix.

\subsection{Operator norm}
In \Cref{subsec:param_projkoop} and for our proposed algorithm, we try
to match the projected Koopman operator $\ProjKoopdDat^t(\Param)$ to
the EDMD operator $\EDMDop$ by minimising the distance between their
matrix representations under the Frobenius norm.  Another natural
choice would be to match the two operators under the operator norm
induced by the inner product on $\FunSpGFin$.  For a linear operator
$\mathcal{A}:\FunSpGFin\to\FunSpGFin$, the operator norm with respect
to the $L^2(\MeaPrDat)$ inner product is given by
\begin{equation}
  \norm{\mathcal{A}}_{\MeaPrDat} = \sup_{g\neq
    0}
  \sqrt{\frac{\int_\MeaSp{[\mathcal{A}g(x)]}^2\,\dd\MeaPrDat(x)}{\int_\MeaSp{[g(x)]}^2\,\dd\MeaPrDat(x)}}.
\end{equation}
Likewise, for a matrix $A\in\mathbb{R}^{\LimBasis\times \LimBasis}$,
the matrix norm with respect to the $\ell_2$ inner product on
$\mathbb{R}^\LimBasis$, weighted by a positive definite matrix
$M\in\mathbb{R}^{\LimBasis\times \LimBasis}$, is given by
\begin{equation}
  \norm{A}_M = \sup_{c\neq
    0}\sqrt{\frac{\Transpose{c}\Transpose{A}MAc}{\Transpose{c}Mc}}.
\end{equation}

As $\FunSpGFin$ is finite dimensional, $\mathcal{A}$ has a matrix
representation $A$ in the basis $\BBasis$, such that
$\mathcal{A}g=\Transpose{c_g}A\BBasis$.  Thus, the operator norm
reduces to
\begin{equation}\label{eq:op_norm_sym}
  \norm{\mathcal{A}}_{\MeaPrDat}
  = \sup_{c_g\neq
    0}\sqrt{\frac{\Transpose{c_g}A\left(
        \int_\MeaPr\BBasis\Transpose{\BBasis}\,\dd\MeaPrDat
      \right)\Transpose{A}c_g}{\Transpose{c_g}\left( \int_\MeaPr\BBasis\Transpose{\BBasis}\,
        \dd\MeaPrDat \right)c_g}}
  = \norm{\Transpose{A}}_{\MassmatDat} =  \norm{A}_{\MassmatDat},
\end{equation}
where the norm of $A$ in equation \eqref{eq:op_norm_sym} is equal to
its transpose as the mass matrix is symmetric. For the projected
Koopman and EDMD operators, we thus have
\begin{equation}
  \norm{\ProjKoopd{\MeaPrDat}^t(\Param)-\EDMDop}_{\MeaPrDat}=\norm{ {\KoopdmatDat(\Param)}-{\EDMDmat} }_{\MassmatDat}.
\end{equation}
Thus, a potentially more natural approach than minimising the
Frobenius norm in~\eqref{eq:min_koop_edmd} and \Cref{alg:edmdparam}
could be find the solution to the minimisation problem
\begin{equation}\label{eq:min_koop_edmd_op}
  \min_{\Param\in\ParamSet}\norm{ {\KoopdmatDat(\Param)} -
    {\EDMDmat} }_{\MassmatDat}.
\end{equation}
Objective evaluations of~\eqref{eq:min_koop_edmd_op} are more
expensive than the Frobenius norm, however it does not yield any
better parameter estimates: We have compared the two methods for the
numerical experiments in this article, and neither method particularly
dominates the other.

\subsection{Constrained EDMD}
Instead of calculating the EDMD operator and matching the Koopman
operator against that, we could also try to do parameter estimation by
constraining the EDMD matrix minimisation in~\eqref{eq:edmd_mat_min}
so that the matrix $A$ is of the form
$\KoopdmatDat=\exp\left( {t\Lmat{\MeaPrDat}\Massmat{\MeaPrDat}^{-1}}
\right)$.  This yields the optimisation problem
\begin{equation}\label{eq:constrained_edmd}
  \min_{\Param\in\ParamSet}
  \norm{\KoopdmatDat(\Param)\BBasis(\Xv)-\BBasis(\Yv)}_F^2.
  =
  \min_{\Param\in\ParamSet}\sum_{j=1}^\LimDat\norm{
    \KoopdmatDat(\Param)\BBasis(x_j)-\BBasis(y_j)}_2^2.
\end{equation}
The number of floating point operations required to evaluate this norm
grows linearly with the amount of data, and hence it becomes expensive
to perform parameter estimation with large amounts of data.  To
compare, in our proposed algorithm, the evaluation of the norm is
independent of data size.

To be sure, constrained EDMD parameter estimation with the
objective~\eqref{eq:constrained_edmd}, can provide good results.  In
\Cref{tbl:ou_comparison_alg_cedmd}, we provide convergence statistics
that compares \Cref{alg:edmdparam} with constrained EDMD.  The table
reports the root mean squared error with three RBFs, from the
\num{2000} Ornstein-Uhlenbeck sample paths from
\Cref{subsec:ouproc_example_conv}.  The parameter estimates are
slightly better for constrained EDMD, and in particular it improves
the convergence rate for $\Param_3$: The error decreases at approximately
$\mathcal O(\LimDat^{-1/2})$, compared to approximately
$\mathcal O(\LimDat^{-1/3})$ with \Cref{alg:edmdparam}.
\begin{table}[htb]
  \centering
  \begin{tabular}{l*{3}{ S[table-auto-round,table-format=-1.3]
    S[table-auto-round,table-format=-1.3]}}
    $j$&\multicolumn{2}{c}{$\Param_1$}&\multicolumn{2}{c}{$\Param_2$}&\multicolumn{2}{c}{$\Param_3$}\\
    \cmidrule(lr){2-3}\cmidrule(lr){4-5}\cmidrule(lr){6-7}
       &{Alg~\ref{alg:edmdparam}}&{C-EDMD}&{Alg~\ref{alg:edmdparam}}&{C-EDMD}&{Alg~\ref{alg:edmdparam}}&{C-EDMD}\\
    \midrule
    0&01.795252400451206&01.6540922896169966&003.042114709338371&003.31497167701178&0002.1537055996427856&0001.7846620149607442\\
    1&01.0701055940286909&00.9833207847963372&001.9763231985747937&001.766966150035098&0001.7386586877479785&0001.281753261861992\\
    2&00.6557251890470801&00.6101505659611432&001.4026972090636629&001.2163813922847412&0001.4611828906183554&0000.9459220081085587\\
    3&00.4356973761683581&00.41702588950718746&000.9873772257111248&000.8399131006612874&0001.160673292172601&0000.6617359184724565\\
    4&00.2926294621123515&00.28175001882171175&000.6846165677308837&000.5857405425450579&0000.8735489571369989&0000.45927381688653287\\
    5&00.204918854296469&00.19640420126071555&000.4762921592012875&000.4067057315495868&0000.6411365285770113&0000.3230315322906311\\
    6&00.14338036368806838&00.13656536839571822&000.34352880521824236&000.2927006825174069&0000.4781571197691795&0000.23296587893921876\\
    7&00.10322745547191186&00.09662434083621661&000.24093350255647923&000.20888289120285226&0000.36238080089434106&0000.1749822995512159\\
    8&00.05254012934856377&00.04617196578972231&000.12194588399186104&000.1021113186577314&0000.21119389457695207&0000.08881784882560835\\
    \midrule
       &{$\times 10^{-1}$}&{$\times 10^{-1}$}&{$\times 10^{-2}$}&{$\times 10^{-2}$}&{$\times 10^{-3}$}&${\times 10^{-3}}$\\
    \bottomrule
  \end{tabular}
  \caption{Root mean squared error comparing the proposed
    \Cref{alg:edmdparam} to constrained
    EDMD~\eqref{eq:constrained_edmd} for $\LimDat=500\times 2^j$ data
    points.  The results are reported from \num{2000} sample paths of
    the OU process in \Cref{subsec:ouproc_example_conv}, with three
    RBFs as explained in the OU process example.
  }\label{tbl:ou_comparison_alg_cedmd}
  %\todo[inline]{For thesis: include convergence values for MLE as well}
\end{table}

The improvements come at a cost, however. First, the objective
in~\eqref{eq:constrained_edmd} results in ill-conditioning for the
backtracking line search with BFGS and the optimiser diverges. To
prevent this, we had to employ the more costly line search by Hager
and Zhang~\cite{hager2006algorithm}, which requires more gradient
evaluations.  Second, evaluating the objective and gradients are more
expensive, especially for larger amounts of data. Users of the
algorithms should choose between these objectives based on their
computational budget and amount of data.

\subsection{Generalised method of moments}
The method of moments approach to parameter estimation is based on the
knowledge of relationships between parameters of a random variable and
its moments. For example, the mean and variance parameters of a
Gaussian random variable can be matched to the empirical mean variance
from a collection of data. One can further extend this idea to match
the expected value and empirical mean of arbitrary functions defined
on the output space of a random variable~\cite{hansen1982large}.  We
take advantage of the Koopman operator to match expected values to
empirical means using the basis functions of $\FunSpGFin$. The
approach is similar to constrained EDMD, however the sum over data is
taken inside the chosen norm on $\mathbb{R}^{\LimBasis}$, as opposed
to summing over the norm in~\eqref{eq:constrained_edmd}.

Let $x_0$ be a random variable distributed according to some
underlying probability space $(\MeaSp,\SigAlg,\MeaPr)$.  For a fixed
$t>0$, define $Y=X_t^{x_0}$, a random variable induced by the product
measure of the Brownian motion and $\MeaPr$.  For $\gf\in\FunSpGInf$,
the expected value of $\gf(Y)$ is given by
\begin{equation}\label{eq:expected_gY}
  \mathbb{E}[g(Y)]=\int_\MeaSp \mathbb{E}_W\left[\gf(X_t^{x})\right]\,\dd\MeaPr(x)
  = \int_\MeaSp \Koopd^t(\Param)\gf(x)\,\dd\MeaPr(x).
\end{equation}
The method of moments aims to find $\Param\in\ParamSet$ to match the
expected value of $\gf(Y)$ for functions $\gf\in\FunSpGInf$ with the
sample mean $m_{\gf} = \frac{1}{\LimDat}\sum_{j=1}^\LimDat
\gf(y_j)$. For a vector field
$\vgf=(\gf_1,\dots,\gf_r)\in{\FunSpGInf}^r$, define the vector sample
mean $\boldsymbol{m}_{\vgf} = \vecbracks{m_{\gf_1},\dots,m_{\gf_r}}$.
The generalised method of moments is then defined as finding a
solution to
\begin{equation}\label{eq:gen_mm_min}
  \Param^*=\argmin_{\Param\in\ParamSet}
  \norm{\mathbb{E}[\vgf(Y)]-\boldsymbol{m}_{\vgf}},
\end{equation}
for a choice of norm on $\mathbb{R}^r$.

Now, choose $\MeaPr=\MeaPrDat$, and let
$\vgf = \BBasis\in \FunSpGFin^\LimBasis$.  Then,
$\Koopd^t(\Param)\vgf(x) = \KoopdmatDat(\Param)\BBasis(x)$.  We see
from~\eqref{eq:expected_gY} that
\begin{equation}
  \mathbb{E}[\vgf(Y)]= \int_{\MeaPrDat}
  \KoopdmatDat(\Param)\BBasis(x)\,\dd\MeaPrDat(x)
  = \frac{1}{\LimDat}\sum_{j=1}^\LimDat \KoopdmatDat(\Param)\BBasis(x_j).
\end{equation}
If we choose the $\ell_2$ norm on $\mathbb{R}^\LimBasis$, then the
method of moments minimisation~\eqref{eq:gen_mm_min} becomes
\begin{align}\label{eq:gmm_ell2}
  \Param^*&=\argmin_{\Param\in\ParamSet}
            \norm{\frac{1}{\LimDat}\sum_{j=1}^\LimDat
            \KoopdmatDat(\Param)\BBasis(x_j)-\BBasis(y_j)}_2.
\end{align}
Contrast~\eqref{eq:gmm_ell2} to the constrained EDMD
problem~\eqref{eq:constrained_edmd}: The sum over data is taken inside
the norm.

In numerical tests, generalised method of moments with the
$\ell_2$-norm gives a very poor estimation performance, and function
evaluations become expensive as the amount of data increases.  The
first point can potentially be fixed, by changing the inner product on
$\mathbb{R}^\LimBasis$ to be an $\Sigma^{-1}$-weighted $\ell_2$ inner
product.  The most effective choice of the inverse weighting matrix
$\Sigma$ is, according to~\cite{hansen1982large,hurn2007seeing}, given
by
\begin{equation}
  \Sigma_{i,j} = \frac{1}{\LimDat}\sum_{k=1}^\LimDat
  \left( [
    \KoopdmatDat(\Param^*)\Basis_i(x_k)-\Basis_i(y_k)][\KoopdmatDat(\Param^*)\Basis_j(x_k)-\Basis_j(y_k)]\right).
\end{equation}
As we do not know $\Param^*$ in advance, this becomes an iterative
procedure in estimating $\Sigma$ and performing the optimisation.

%%%%%%%%%%%%%%%%%%%%%%%%%%%%%%%%%%%%%
\section{Discussion and Conclusion}\label{sec:conclusion}
%%%%%%%%%%%%%%%%%%%%%%%%%%%%%%%%%%%%%

In this paper, we presented a method to parameterise SDEs based off
approximating the generator. We provided bounds on the mean square
error of the EDMD matrix as an estimator in \Cref{sec:sdeintro}, numerical examples in
\Cref{sec:num_examples}, and suggested variations to the method in
\Cref{sec:variations_algo}. Thus far, our work has been limited to
SDEs, although other models are also of interest. We envision our
method being a suitable starting to point to parameterising a wide
range of stochastic dynamical systems when the generator of the
process is known. The methodology could also be applied to
deterministic systems, although more established methodologies already
exist (e.g., minimising mean squared error).

Our algorithm appears to not fit into the broad MLE or MM categories
for parameter estimation. Therefore, our work opens up new research
directions which we now briefly discuss.

One of the advantages to our method is that once the data matrices
$\BBasis(\Xv)$ and $\BBasis(\Yv)$ are constructed, the parameter
search does not depend on the number of data points $\LimDat$, only
the number of basis functions $\LimBasis$. In the limit of large data
$\LimDat \gg \LimBasis$, the data matrices will be computationally
intensive to construct, so we hope to sub-sample the data and compute
these efficiently. Additionally, there are alternatives to Monte Carlo
sampling of the generator matrix $\Lmat{\MeaPr}$. For example, one
could use kernel density estimation to find $\MeaPr$, and use
numerical integration to calculate the matrix entries.

Numerical experiments show that the method performs as well as a wide
range of existing methods. In \Cref{subsec:ouproc_example_conv} we
find that the convergence rate of the parameter estimation errors
decrease by orders $\mathcal{O}(\LimDat^{-1/2})$ and
$\mathcal{O}(\LimDat^{-1/3})$, perhaps indicating that accelerating
ideas from Monte Carlo approximations can improve convergence.  In the
numerical example in \Cref{sec_num_example_1}, we investigated
prediction accuracy whilst varying numbers of basis functions and
numbers of eigenfunctions in the approximation. Occasionally it was
the case that a limited eigen-expansions of the Koopman operator was
preferable to the full matrix. It is not clear when a limited
eigen-expansion is preferable to the full matrix.

When considering models with many parameters, there are many issues
around the topics of model selection, confidence in parameters, and
sensitivity analysis \cite{hamby1994review, kadane2004methods,
  zhang2015selection}. This is especially the case in our work when
many SDEs correspond to the same infinitesimal generator\footnote{This
  non-uniqueness arises from SDEs with multiple noise terms.}
\cite{schnoerr2014complex}. Theoretical advancements relating to rates
of convergence, especially with regards to error analysis, are now of
critical importance to promote the use of our method. We also see the
need to test our method on high dimensional stochastic dynamical
systems, especially ones in which diverse ranges of behaviour are
possible.

%%%%%%%%%%%%%%%%%%%%%%%%%%%%%%%%%%%%%
\section*{Contributions}
%%%%%%%%%%%%%%%%%%%%%%%%%%%%%%%%%%%%%
J.P.T-K and A.N.R contributed equally to this article. J.P.T-K had the
idea of EDMD-based parameter estimation, and performed the numerical
experiments with lower order eigenfunction expansions.  A.N.R devised
the proposed algorithm and the theoretical connection to Koopman
theory, performed the numerical experiments on the OU process, and
formalised the variations of the algorithm.

%%%%%%%%%%%%%%%%%%%%%%%%%%%%%%%%%%%%%
\section*{Acknowledgements}
%%%%%%%%%%%%%%%%%%%%%%%%%%%%%%%%%%%%%

A.N.R and J.P.T-K both received funding from the EPSRC under grant
reference numbers EP/L015803/1 (A.N.R), and EP/G037280/1 (J.P.T-K). We
thank Stefan Klus, P\'{e}ter Koltai, Scott Dawson, and in particular
Tomislav Plesa, for helpful discussions.

%%%%%%%%%%%%%%%%%%%%%%%%%%%%%%%%%%%%%
\bibliographystyle{siamplain} \bibliography{references}
%%%%%%%%%%%%%%%%%%%%%%%%%%%%%%%%%%%%%
\end{document}